\documentclass[11pt]{article}

\usepackage[margin=1in]{geometry}
\usepackage{amsmath,amssymb,amsthm,mathtools}
\usepackage[T1]{fontenc}
\usepackage[utf8]{inputenc}
\usepackage{lmodern}
\usepackage{enumitem}
\usepackage{hyperref}
\usepackage{xcolor}
\usepackage{graphicx}
\usepackage{booktabs}
\usepackage{float}

\newtheorem{lemma}{Lemma}[section]
\newtheorem{proposition}[lemma]{Proposition}
\newtheorem{theorem}[lemma]{Theorem}
\theoremstyle{definition}
\newtheorem{definition}[lemma]{Definition}
\newtheorem{remark}[lemma]{Remark}

\newcommand{\RR}{\mathbb{R}}
\newcommand{\EE}{\mathbb{E}}
\newcommand{\FF}{\mathbb{F}}
\newcommand{\cA}{\mathcal{A}}

\newcommand{\cF}{\mathcal{F}}

\newcommand{\TV}{\operatorname{TV}}
\newcommand{\dd}{\,d}
\newcommand{\ellbar}{\underline{\ell}}

\title{A Singular Control Problem for Data Center Electricity Cost Minimization}
\author{Ren\'e Carmona and Xiaoyu Cui}

\begin{document}
\maketitle

\begin{abstract}
The goal of the paper is to provide a complete analysis of a stochastic control problem when the cost to minimize involves the running maximum of the underlying controlled state process. The model is motivated by the electricity cost minimization
of a data center facing coincident peak charges and possibly benefiting from participation in a demand-response
program. From a mathematical standpoint, the main challenge comes from the inclusion of the running maximum in the state
of the system forcing the original problem into a singular stochastic control problem on a closed wedge in the plane
with reflection on parts of the boundary.
We prove that the value function is locally Lipschitz continuous in space using a
gap-monotonicity estimate for coupled reflected loads and a deterministic
total-variation bound for the difference of two running maxima.
\vskip 0pt
The dynamic programming principle and viscosity formulation are then
derived for ordinary progressively measurable controls. The corresponding HJB equation includes a reflecting Neumann condition on the lower boundary for the load direction as well as an oblique boundary condition on the diagonal part of the boundary of the domain due to the effect of the running-maximum process. We prove existence of a solution in the viscosity sense as well as a comparison theorem providing  conditional uniqueness in the corresponding viscosity class.
\vskip 0pt
Finally we provide numerical results illustrating the trade-off between the two conflicting incentives comprising the expected cost of the data center.
\end{abstract}


\section{Introduction}

Data centers have become one of the fastest-growing sources of electricity demand in the world, and their load is heavily concentrated in a small number of regional grids. The International Energy Agency estimates that global electricity consumption by data centers is on track to roughly double between 2024 and 2030, with artificial-intelligence workloads accounting for most of the increase \cite{IEA2026}; independent studies point in the same direction \cite{demand_boom}. Because this new demand is spatially concentrated rather than spread uniformly across the grid, even a moderate aggregate increase can translate into acute local stress on transmission and distribution capacity, and into material costs for the system operators and utilities that must accommodate it \cite{Microsoft_review}. This has led system operators to lean increasingly on capacity-based charges, such as coincident-peak (CP) and non-coincident-peak (NCP) demand charges, and on demand-response (DR) programs, to shape the load of large flexible consumers such as data centers \cite{DataCenterDR1,DataCenterDR2}. The rapid buildout of data center capacity has also drawn scrutiny for reasons that go beyond the electricity bill itself, including its water consumption for cooling and the noise generated by cooling equipment and backup generation, both of which have become recurring sources of concern for regulators and host communities \cite{DataCenterNoise}. Providing economists, utilities, and policymakers with a rigorous, tractable model of how a single data center actually responds to these price signals and incentive programs is therefore of interest well beyond the mathematical questions the model raises, and it is the practical motivation for the present paper.

\vskip 2pt
With this motivation in mind, we study a stylized, but not unrealistic, model of the electricity-consumption decisions of a single data center. Being stylized, the model deliberately leaves out many features of an actual utility bill, such as the baseline cost of serving the existing load and the data center's own revenue streams; being grounded in the practice of grid operations, it keeps the two features which actually drive the tension in a data center's response to its environment, namely capacity-style demand charges of the CP/NCP type and DR incentive payments, both of which are real mechanisms routinely used by Independent System Operators (ISOs) and utilities. Formally, this note studies a single-agent continuous-time stochastic control problem over a one-dimensional controlled diffusion process with a lower reflecting boundary. We denote this process $(L_t)_{0\le t\le T}$. It is intended to represent the electricity demand (load) at time $t$ of a data center trying to minimize the cost of its electricity consumption over the period $[0,T]$.

A major mathematical difficulty comes from the fact that on top of the price of the power actually consumed, the data center faces coincident peak charges which can be offset or compounded by its participation in a demand response program. The coincident peak charges are only known long after $T$, so they are not "non-anticipative" like the typical inputs of a standard stochastic control problem. For modeling purposes, we include their impact via an integral with respect to the running maximum $M_s$ of the load $L_t$ for $t\le s$. While not a perfectly realistic model, this proxy has proven to be extremely useful in practice. See for example \cite{Carmona_CP} and \cite{Carmona_battery_opt}. In order to keep the problem in the realm of Markovian stochastic control problems, we include the running maximum in the state of the system.
So for this reason, the underlying state of our optimization problem
is $(L_s,M_s)$, where $L_s$ is the load process and $M_s=\sup_{u\le s}L_u$ its running maximum. Visualizing the time evolution of the pair $(L_s,M_s)$ in the $(\ell,m)$-plane is enlightening: the trajectories move horizontally on a line on which $m$ is constant as long as $\ell<m$, and there is a motion in the $m$-direction only along the diagonal $\ell=m$.

\vskip 2pt
This dependence of the cost on the running maximum of the state, together with the competition it creates against the demand-response reward, is the main source of originality of the model, and it is also what makes the problem mathematically demanding: publications on stochastic control models involving the running maximum of the state are few and far between. \cite{RMaxControl0} is the earliest publication we found tackling the optimal control of a diffusion process with a running cost depending on the running maximum of the diffusion, and using the pair $(L_s,M_s)$ as the new state for the analysis, arguing that while $M_t$ alone is not a Markov process, the pair $(L_s,M_s)$ is a strong Markov process when the optimization is restricted to closed loop Markovian feedback controls. While identifying the oblique boundary condition for the HJB equation, only particular cases are solved. In \cite{RMaxControl1} and \cite{RMaxControl2} the infinite horizon discounted cost minimization problem is addressed and a classical solution of the HJB equation is provided. None of these references combine the running-maximum penalty with a lower reflecting constraint on the state, a finite time horizon, and a demand-response reward, which is exactly the combination of features created by our data-center application, and which had, to the best of our knowledge, not yet been studied systematically.

\vskip 4pt
For physical reasons, we impose a lower operational bound \(\underline\ell\) on the load and, for definiteness, assume that the load is reflected when it reaches this lower boundary. The present formulation does not include an upper reflecting boundary; adding one would introduce an additional boundary condition in the HJB system.

The thrust of the present paper is to identify the dynamic programming and viscosity structure of this running-maximum control problem, and to do so rigorously. The coincident-peak term is written in the Stieltjes form \(\kappa\int\varphi'(L_s)\,dM_s\), but, because \(M\) increases only when \(L=M\), it is equivalent to the terminal peak penalty \(\kappa(\varphi(M_T)-\varphi(m))\). This avoids the representative ambiguity which would arise if one tried to integrate a usual \(dt\)-control against the singular measure \(dM\). Building on this reduction, our mathematical contributions are: (i) a rigorous derivation of the dynamic programming principle for the augmented state \((L,M)\); (ii) the viscosity formulation of the corresponding Hamilton-Jacobi-Bellman (HJB) equation, which includes a standard Neumann condition at the lower reflecting boundary together with a non-standard, oblique diagonal condition \(v_m+\kappa\varphi'(m)=0\) on the running-maximum boundary, reflecting the singular nature of the peak charge; and (iii) a comparison principle, proved after transforming the wedge problem into an oblique derivative problem on the orthant, which yields existence and conditional uniqueness of the value function in the corresponding viscosity class. This last step is carried out in the same spirit as in the paper \cite{RMaxControl4}, where a running-maximum cost is studied for a diffusion on the entire line, and \cite{RMaxControl5}, where the infinite horizon analog is considered; the present formulation differs by including a finite horizon, a lower reflecting load constraint, the data-center running cost, and the demand-response term. Finally, we complement the theoretical analysis with detailed numerical experiments, reported in Section~\ref{sec:numerical-experiments-draft} below, which illustrate the trade-off between the coincident-peak penalty and the demand-response reward and the resulting shape of the optimal feedback control.

\section{The Reflected Running-Maximum Control Problem}

We fix the time horizon $T>0$ and we let $(\Omega,\mathcal F,\FF=(\cF_t)_{0\le t\le T},\mathbb P)$ be a filtered probability space satisfying the usual conditions and supporting a Brownian motion $W=(W_t)_{0\le t\le T}$. The control processes take values in the set $A=[\underline\delta,\overline\delta]$ with $\underline\delta<0<\overline\delta$, and we denote by $\cA$ the set of $\FF$-progressively measurable processes taking values in $A$. For a starting time \(t\), \(\cA_t\) denotes the corresponding class of admissible controls on \([t,T]\).

\subsection{Underlying state process}
We introduce the basic reflected load process.
Let $\Theta:[0,T]\to \RR$ be a bounded continuous function. We can think of $\Theta$ as the average seasonal load component  toward which the electricity demand of the data center mean-reverts.  
For $\delta\in\cA$, $t\in[0,T]$ and $\ell\geq \ellbar$, the evolution of the load starting from $L_t=\ell$ at time $t\in[0,T]$ is given by:
\begin{equation}
\label{fo:load_dynamics}
dL^{t,\ell}_s=\bigl[-\lambda(L^{t,\ell}_s-\Theta_s)+\delta_s\bigr]\dd s+\sigma\dd W_s+\dd K^{t,\ell}_s,\qquad L^{t,\ell}_t=\ell,
\end{equation}
where $\lambda$ and $\sigma$ are positive constants, and $K^{t,\ell}$ is the adapted non-decreasing process reflecting the load at its lower boundary $\ellbar$ so that $L^{t,\ell}_s\geq \ellbar$ for all $s\ge t$, 
\[
K^{t,\ell}_t=0,
\qquad \text{and}\quad
\int_t^T \mathbf 1_{\{L_s>\ellbar\}}\dd K^{t,\ell}_s=0,
\]
and so that $s\mapsto K^{t,\ell}_s$ increases only when $L^{t,\ell}_s=\ellbar$.
These processes, as well as those introduced below, can be assumed to be defined on the whole interval \([0,T]\) by setting them constant equal to their value at time \(t\) on the interval \([0,t]\).
They can be constructed using the one-dimensional Skorokhod map. More precisely, for a candidate load path \(L\), set
\[
\xi_s^L=\ell+\int_t^s\bigl[-\lambda(L_u-\Theta_u)+\delta_u\bigr]\,du
+\sigma(W_s-W_t),
\]
and require \((L,K)=\mathcal S(\xi^L)\), where \(\mathcal S\) denotes the lower Skorokhod map. In this notation the minimal lower reflection is given pathwise by
\[
K_s=\sup_{t\le u\le s}(\ellbar-\xi_u^L)^+,\qquad
L_s=\xi_s^L+K_s.
\]
This is an implicit fixed-point representation of the reflected SDE; it is not an autonomous unreflected Ornstein--Uhlenbeck equation. Standard Lipschitz estimates for the one-dimensional Skorokhod problem give existence and pathwise uniqueness; see, for example, \cite{LionsSznitman1984}.

\vskip 4pt\noindent
Next, we introduce the running maximum process.
For $m\geq \ell$, we define the running maximum by
\begin{equation}
\label{fo:rmax}
M^{t,\ell,m}_s=\max\{m,\sup_{u\in[t,s]}L^{t,\ell}_u\}.
\end{equation}
$M^{t,\ell,m}$ is continuous, nondecreasing, and increases only on the set $\{L^{t,\ell}_s=M^{t,\ell,m}_s\}$.
The state space of the pair $(L,M)$ is the closed wedge
\[
D=\{(\ell,m):\ellbar\leq \ell\leq m\}.
\]
When the dependence on the control is relevant, we write \(L^{t,\ell,\delta}\) and \(M^{t,\ell,m,\delta}\).

\subsection{Cost Functional and Value Function}
We drop (at least momentarily) the superscripts to lighten the notations.
The overall cost faced by the data center is the aggregate of several sources which we introduce separately.
The first one is given by the (short)  running cost function $f$ defined as:
\begin{equation}
\label{fo:short_cost}
f(t,\ell,\delta)=(a+b\ell-\alpha_0)\delta+\frac{\beta}{2}\delta^2
\end{equation}
for some fixed parameters \(a\in\RR\), \(b\geq0\), \(\beta>0\), and a fixed offset or marginal rebate parameter \(\alpha_0\in\RR\). We interpret \(f\) as a reduced-form marginal adjustment cost for the controlled load, rather than a full accounting model for the electricity bill of the data center. In particular, the baseline cost of serving the existing load and the main revenue sources of the data center are not modeled explicitly. This convention is useful for isolating the interaction between the control effort, the CP penalty, and the DR incentive in our stylized stochastic control problem. The second source is the price to pay for a typical \emph{Coincident Peak} (CP) program of the Independent System Operator (ISO) of the region in which the data center operates. We model it as
\begin{equation}
\label{fo:peak_cost}
\kappa\int_t^T \varphi'(L^{t,\ell,\delta}_s)\dd M^{t,\ell,m,\delta}_s.
\end{equation}
for some positive constant $\kappa>0$, and a peak-penalty function \(\varphi:[\ellbar,\infty)\to\RR\) which is assumed to be convex, nondecreasing, \(C^2\), and with at most polynomial growth.
It penalizes the high values of the load over the period, with a charge paid each time a new peak is created.  We assume \(\varphi'(\ellbar)\ge0\), which gives \(\varphi'\ge0\) on \([\ellbar,\infty)\) because of our convexity assumption. The two cases motivating the discussion are \(\varphi(r)=r\), which gives a linear peak charge, and \(\varphi(r)=r^2/2\), which gives the load-weighted penalty \(\kappa\int L_s\,dM_s\). 
Pathwise, the integral with respect to the running maximum is a Stieltjes integral, so since \(L\) is continuous we have the following simplification.

\begin{proposition}[Reduction of the peak Stieltjes term]\label{pr:peak-reduction-v6}
For every admissible control \(\delta\) and every initial state \((\ell,m)\in D\), almost surely
$$
\int_t^T \varphi'(L_s^{t,\ell,\delta})\dd M_s^{t,\ell,m,\delta}
=\int_t^T \varphi'(M_s^{t,\ell,m,\delta})\dd M_s^{t,\ell,m,\delta}
=\varphi(M_T^{t,\ell,m,\delta})-\varphi(m).
$$
\end{proposition}

\begin{proof}
The running maximum \(M\) is continuous, nondecreasing, and increases only on the contact set \(\{L=M\}\). Hence \(L_s=M_s\) for \(dM_s\)-a.e. \(s\), and the first equality follows. Since \(M\) has finite variation and \(\varphi\in C^1\), the pathwise change-of-variables formula for finite-variation paths gives
\[
d\varphi(M_s)=\varphi'(M_s)\,dM_s .
\]
Integrating from \(t\) to \(T\), and using \(M_t=m\), gives the identity.
\end{proof}

Thus from a purely mathematical viewpoint, the CP cost can be interpreted as a terminal cost \(\kappa\varphi(M_T)\), up to the initial-state constant \(-\kappa\varphi(m)\), over the enlarged state process $(L,M)$. However, we keep the Stieltjes integral formulation of the cost because it records the economic interpretation: the charge is accumulated when the load creates a new demand peak. In estimates and in the numerical implementation, however, it is legitimate and often simpler to use the terminal representation.

\vskip 2pt
We shall also include in our analysis the possibility for the data center to be enrolled in a demand-response program. The cost (which is in fact a reward because of the minus sign) corresponding to this third source is given by:
\begin{equation}
\label{fo:DR_cost}
\int_t^T\Pi_s(B_s-L_s)^+ds
\end{equation}
where the deterministic function \(t\mapsto B_t\) is a bounded continuous baseline target chosen by the Independent System Operator (ISO), and \(t\mapsto \Pi_t\ge0\) is a bounded continuous incentive rate.
This third source in the expected cost is an incentive for the data center to keep its load below the baseline target when the rate $\Pi_t$ is nonzero.

\vskip 4pt
So the data center will attempt to minimize its overall expected cost, namely the objective 
\begin{equation}
\label{fo:full_cost}
J(t,\ell,m;\delta)=\EE\left[\int_t^T f(s,L_s,\delta_s)\dd s+
\kappa\int_t^T \varphi'(L_s)\dd M_s-\int_t^T\Pi_s(B_s-L_s)^+ds\right],
\end{equation}

The value function of the optimization problem is thus:
\begin{equation}
\label{fo:value}
v(t,\ell,m)=\inf_{\delta\in\cA_t}J(t,\ell,m;\delta)\qquad
\text{with}\quad v(T,\ell,m)=0.
\end{equation}

\section{Regularity of the State Process}

In this section, we gather some of the properties of the load process we shall use in the sequel. We denote by $(L^{t,\ell,\delta}_s)_{t\le s\le T}$ the load process starting from $L_t=\ell$ at time $t$ and controlled by $\delta\in\cA$. So we add the control process $\delta$ to the superscript to emphasize its role and impact.

\begin{proposition}
\label{pr:load}
(Load monotonicity) Under ordered initial data $\ell\leq\ell'$, with the same control $\delta$, and the same Brownian motion,
\begin{equation}
\label{fo:load_monotonicity}
0\leq L_s^{t,\ell',\delta}-L_s^{t,\ell,\delta}\leq (\ell'-\ell)e^{-\lambda(s-t)}.
\end{equation}
which implies that 
\begin{equation}
\label{fo:load_Lip}
\sup_{t\le s\le T} |L_s^{t,\ell,\delta}-L_s^{t,\ell',\delta}|\leq |\ell-\ell'|.
\end{equation}
\end{proposition}

\begin{proof}
Since we use the same $\delta\in\cA$ and the same Brownian path, if we set $\Delta_s=L^{t,\ell',\delta}_s-L^{t,\ell,\delta}_s$ for $s\in [t,T]$, standard properties of the Skorokhod map give
$$
L_s^{t,\ell,\delta}\le L_s^{t,\ell',\delta}\qquad\text{and}\qquad K_s^{t,\ell',\delta}\le K_s^{t,\ell,\delta}
$$
for all \(s\in[t,T]\). This is the standard monotonicity of the one-dimensional lower Skorokhod map; equivalently, applying the comparison theorem for reflected equations gives \(dK_s^{t,\ell',\delta}\le dK_s^{t,\ell,\delta}\) as measures. Subtracting the two reflected SDEs gives
\[
d\Delta_s=-\lambda\Delta_s\,ds+dK_s^{t,\ell',\delta}-dK_s^{t,\ell,\delta}
\le -\lambda\Delta_s\,ds,
\]
and Gronwall's lemma yields \(\Delta_s\le(\ell'-\ell)e^{-\lambda(s-t)}\). This proves \eqref{fo:load_monotonicity}. Since \(e^{-\lambda(s-t)}\le1\), \eqref{fo:load_Lip} follows directly.
\end{proof}

The proof of the following proposition will use the simple estimate below.

\begin{lemma}
\label{le:TV}
Let us assume that \(x\) and \(h\) are deterministic continuous functions on an interval \([a,b]\), with \(h_s\geq0\) non-increasing and bounded by \(d=\|h\|_\infty\), and let \(m\ge x_a\). Then if we define \(N_s\) for \(s\in[a,b]\) by
\[
N_s = m\vee\sup_{a\leq u\leq s}(x_u+h_u)-m\vee\sup_{a\leq u\leq s}x_u,
\]
we have \(0\leq N_s\leq d\) and its total variation over the interval $[a,b]$ satisfies \(\TV_{[a,b]}(N)\leq2d\).
\end{lemma}

\begin{proof}
Set
\[
A_s=m\vee\sup_{a\leq u\leq s}x_u,\qquad
B_s=m\vee\sup_{a\leq u\leq s}(x_u+h_u).
\]
Obviously \(N_s=B_s-A_s\ge0\). Moreover \(B_s\le A_s+d\), so \(0\le N_s\le d\). Let
\[
\tau=\inf\{s\in[a,b]:x_s>m\},
\]
with \(\tau=b\) if the set is empty. On \([a,\tau]\), \(A_s=m\), and therefore
\[
N_s=B_s-m=m\vee\sup_{a\le u\le s}(x_u+h_u)-m
\]
is non-decreasing. Its total increase on this interval is at most \(d\), since \(0\le N_s\le d\).
If the set \(\{s\in[a,b]:x_s>m\}\) is empty, then \(\tau=b\), \(A_s\equiv m\) on the whole interval, and this already gives \(\TV_{[a,b]}(N)\le d\).

It remains to control the variation on \([\tau,b]\). On this interval \(A_s=\sup_{a\le u\le s}x_u\). The set
\[
O=\{s\in[\tau,b]:x_s<A_s\}
\]
is open in \([\tau,b]\), and can be written as the countable union of disjoint open intervals
\[
O=\bigcup_i(s_i,t_i).
\]
On each such interval, \(A\) is constant. By continuity, at every finite left endpoint \(s_i\) one has \(x_{s_i}=A_{s_i}\); if \(s_i=\tau\), this also follows from the definition of \(\tau\) and continuity, since \(A_\tau=m=x_\tau\) unless the set \(\{x>m\}\) is empty. Hence, for \(u\in[s_i,t_i]\),
\[
x_u+h_u\le A_{s_i}+h_{s_i}=x_{s_i}+h_{s_i}\le B_{s_i},
\]
where we used that \(h\) is non-increasing. Thus \(B\), and therefore \(N=B-A\), is flat on each interval \((s_i,t_i)\).

On the complement of \(O\), the path \(x\) is at its running maximum. If \(r_1\le r_2\) are two points of this complement, then \(A_{r_1}=x_{r_1}\), \(A_{r_2}=x_{r_2}\), and
\[
\begin{aligned}
N_{r_2}-N_{r_1}
&=B_{r_2}-A_{r_2}-(B_{r_1}-A_{r_1})\\
&=B_{r_2}-B_{r_1}+x_{r_1}-x_{r_2}\\
&=\max\Bigl(0,\sup_{r_1\le u\le r_2}(x_u+h_u)-B_{r_1}\Bigr)+x_{r_1}-x_{r_2}\\
&\le \max\Bigl(0,\sup_{r_1\le u\le r_2}(x_u-x_{r_1})\Bigr)+x_{r_1}-x_{r_2}\\
&\le \max(0,x_{r_2}-x_{r_1})+x_{r_1}-x_{r_2}\\
&\le0.
\end{aligned}
\]
Thus \(N\) is flat on the intervals where \(x\) is strictly below its running maximum and is non-increasing on the complementary pieces. Consequently \(N\) is non-increasing on \([\tau,b]\). Since \(N\) remains in \([0,d]\), the total variation on \([a,\tau]\) is at most \(d\), and the total variation on \([\tau,b]\) is at most \(d\). Hence \(\TV_{[a,b]}(N)\le2d\).
\end{proof}

\begin{proposition}
\label{pr:rmax}
\begin{itemize}
\item[(i)] (Running-maximum regularity)
\begin{equation}
\label{fo:rmax_Lip}
\sup_s |M_s^{t,\ell,m,\delta}-M_s^{t,\ell',m',\delta}|
\leq |m-m'|+\sup_s |L_s^{t,\ell,\delta}-L_s^{t,\ell',\delta}|.
\end{equation}
\item[(ii)] (Running-maximum TV stability) Fix the same control $\delta\in\cA_t$ and the same Brownian motion. For any two initial states $(\ell,m),(\ell',m')\in D$,
\begin{equation}
\label{fo:M_TV}
\TV_{[t,T]}\bigl(M^{t,\ell',m',\delta}-M^{t,\ell,m,\delta}\bigr)
\leq 2|\ell'-\ell|+|m'-m|.
\end{equation}
\end{itemize}
Consequently,
the terminal running maxima are Lipschitz in the initial data:
\[
|M_T^{t,\ell',m',\delta}-M_T^{t,\ell,m,\delta}|
\leq |\ell'-\ell|+|m'-m|.
\]
\end{proposition}

\begin{proof}
\eqref{fo:rmax_Lip} follows from the fact that $|a\vee x-b\vee y|\le |a-b|+|x-y|$ and the definition \eqref{fo:rmax} of the running maximum. We now prove \eqref{fo:M_TV}.
If \(m=m'\) and \(\ell\le\ell'\), we use Lemma \ref{le:TV} with $x_s=L^{t,\ell,\delta}_s$ and \(h_s=L^{t,\ell',\delta}_s-L^{t,\ell,\delta}_s\), which is nonnegative and nonincreasing by Proposition \ref{pr:load}, together with the bound \eqref{fo:load_Lip}. The same conclusion holds if \(\ell'\le \ell\), after exchanging the two loads. For general \(m,m'\), set \(\bar m=m\vee m'\). Since \((\ell,m)\) and \((\ell',m')\) belong to \(D\), the intermediate pairs \((\ell,\bar m)\) and \((\ell',\bar m)\) also belong to \(D\). Then
\begin{equation*}
\begin{split}
\TV\bigl( M^{t,\ell',m',\delta}-M^{t,\ell,m,\delta}\bigr) 
&\leq \TV\bigl(M^{t,\ell',m',\delta}-M^{t,\ell',\bar m,\delta}\bigr)
+\TV\bigl(M^{t,\ell',\bar m,\delta}-M^{t,\ell,\bar m,\delta}\bigr)\\
&\quad+\TV\bigl(M^{t,\ell,\bar m,\delta}-M^{t,\ell,m,\delta}\bigr)\\
&\leq |\bar m-m'|+2|\ell-\ell'|+|\bar m-m|\\
&=|m'-m|+2|\ell-\ell'|.
\end{split}
\end{equation*}
Here we used the equal-initial-maximum case for the middle term. For the first and third terms, with the same load and two different initial maxima, the difference of the corresponding running maxima is monotone between its initial value and \(0\), so its total variation is bounded by the difference of the two initial maxima. The terminal Lipschitz estimate follows from \eqref{fo:rmax_Lip} and Proposition \ref{pr:load}.
\end{proof}

\begin{lemma}[Short-time state estimate]\label{le:short-time-state-v5}
Let \(K\subset D\) be compact. There is a constant \(C_K\), depending only on \(K\) and the model data, such that for all \((\ell,m)\in K\), all \(t\le t'\), and all bounded admissible ordinary controls \(u\),
\[
\EE\left[\sup_{t\le s\le t'}|L_s^{t,\ell,u}-\ell|+M_{t'}^{t,\ell,m,u}-m\right]\le C_K |t'-t|^{1/2}.
\]
\end{lemma}

\begin{proof}
The one-dimensional Skorokhod map is Lipschitz on continuous paths. Applying this property to the driving path of the reflected equation gives, for \(h=t'-t\),
\[
\sup_{t\le s\le t'}|L_s^{t,\ell,u}-\ell|
\le C_K h+C_K\int_t^{t'}\sup_{t\le r\le s}|L_r^{t,\ell,u}-\ell|\,ds
+C_K\sup_{t\le s\le t'}|W_s-W_t|.
\]
The constant \(C_K\) uses the compactness of \(K\), boundedness of \(u\) and \(\Theta\), and the standard first-moment bound for the reflected load on the short interval. Gronwall's lemma, followed by the Burkholder-Davis-Gundy inequality for the Brownian increment, gives
\[
\EE\left[\sup_{t\le s\le t'}|L_s^{t,\ell,u}-\ell|\right]\le C_K|t'-t|^{1/2}.
\]
Since \(m\ge\ell\),
\[
0\le M_{t'}^{t,\ell,m,u}-m
\le \sup_{t\le s\le t'}(L_s^{t,\ell,u}-\ell)^+
\le \sup_{t\le s\le t'}|L_s^{t,\ell,u}-\ell|,
\]
which proves the estimate.
\end{proof}

\section{Regularity of the Value Function}

\begin{proposition}[Spatial Lipschitz continuity]\label{pr:spatial-lip-v5}
For every compact \(K\subset D\), there is a constant \(C_K>0\) such that
\[
|v(t,\ell,m)-v(t,\ell',m')|\leq C_K(|\ell-\ell'|+|m-m'|)
\]
for all \((\ell,m),(\ell',m')\in K\).
\end{proposition}

\begin{proof}
For $\epsilon>0$, let $\delta^\epsilon\in\cA_t$ such that 
$$
J(t,\ell,m,\delta^\epsilon)\le v(t,\ell,m)+\epsilon.
$$
Then
\begin{equation*}
\begin{split}
v(t,\ell',m')-v(t,\ell,m)
&\le J(t,\ell',m',\delta^\epsilon)-v(t,\ell,m)\\
&\le J(t,\ell',m',\delta^\epsilon)-J(t,\ell,m,\delta^\epsilon)+\epsilon\\
&\le C\bigl(|\ell-\ell'|+|m-m'|\bigr) +\epsilon
\end{split}
\end{equation*}
where we used Lemma \ref{le:Y_Lip}. Since $\epsilon>0$ was arbitrary, we conclude that
$$
v(t,\ell',m')-v(t,\ell,m) \le C_K\bigl(|\ell-\ell'|+|m-m'|\bigr)
$$
and reversing the roles of $v(t,\ell',m')$ and $v(t,\ell,m)$ gives the desired result.
\end{proof}

\begin{proposition}[Local time-H\"older continuity]\label{pr:time-holder-v6}
For every compact \(K\subset D\), there exists \(C_K>0\) such that
\[
|v(t,\ell,m)-v(t',\ell,m)|\leq C_K|t'-t|^{1/2}
\]
for \((\ell,m)\in K\).
\end{proposition}

\begin{proof}
Let us assume that \(t\le t'\), and set \(h=t'-t\). We first record the short-time estimate used below. Uniformly over bounded ordinary controls \(\delta\),
\[
\EE\left[\sup_{t\le s\le t'}|L_s^{t,\ell,\delta}-\ell|+M_{t'}^{t,\ell,m,\delta}-m\right]\le C_K |t'-t|^{1/2}.
\]
This is Lemma \ref{le:short-time-state-v5}, and it is the reflected analogue of the standard estimate obtained from the Lipschitz property of the Skorokhod map, Gronwall's lemma, and the Burkholder-Davis-Gundy inequality.
The DPP used below is proved in Section~\ref{sec:dpp} from the usual restriction, pasting, flow, and measurable-selection construction; it does not use the time-continuity estimate proved here.

For \(\varepsilon>0\), let \(\delta^\varepsilon\in\cA_t\) be such that
\[
J(t,\ell,m;\delta^\varepsilon)\le v(t,\ell,m)+\varepsilon .
\]
Applying the DPP of Theorem \ref{thm:dpp-v6} on the deterministic interval \([t,t']\), and using this nearly optimal control, gives the following inequality.
For this display, write
\(L_s^\varepsilon=L_s^{t,\ell,\delta^\varepsilon}\) and
\(M_s^\varepsilon=M_s^{t,\ell,m,\delta^\varepsilon}\).
\begin{align*}
v(t',\ell,m)-v(t,\ell,m)
&\le \EE\bigl[
v(t',\ell,m)-v(t',L_{t'}^\varepsilon,M_{t'}^\varepsilon)
\bigr]\\
&\quad-\EE\left[\int_t^{t'} f(s,L_s^\varepsilon,\delta_s^\varepsilon)\dd s\right]\\
&\quad-\kappa\EE\left[\int_t^{t'}\varphi'(L_s^\varepsilon)\dd M_s^\varepsilon\right]\\
&\quad+\EE\left[\int_t^{t'}\Pi_s(B_s-L_s^\varepsilon)^+\,\dd s\right]
+\varepsilon\\
&=(i)+(ii)+(iii)+(iv)+\varepsilon .
\end{align*}
By the spatial Lipschitz estimate, the first term satisfies
\[
|(i)|\le C_K\EE\bigl[
|L_{t'}^\varepsilon-\ell|
+|M_{t'}^\varepsilon-m|
\bigr]\le C_K h^{1/2}.
\]
The linear growth of \(f\) in the load variable, boundedness of the controls, and the compactness of \(K\), together with the short-time estimate, give
\[
|(ii)|\le C_K h .
\]
Finally,
\[
0\le M_{t'}^\varepsilon-m
\le \sup_{t\le s\le t'}|L_s^\varepsilon-\ell|,
\]
and \(\varphi'\) is bounded on the stopped compact neighborhood, while the polynomial tails are controlled by the standard reflected-SDE moment estimates, so
\[
|(iii)|\le C_K h^{1/2}.
\]
The demand-response term is controlled by the boundedness of \(B\) and \(\Pi\), together with the short-time state estimate, so
\[
|(iv)|\le C_K h .
\]
Putting these estimates together and then letting \(\varepsilon\downarrow0\), we obtain
\[
v(t',\ell,m)-v(t,\ell,m)\le C_K h^{1/2}.
\]

For the opposite inequality, we proceed as in the original argument by using a simple control on the short interval \([t,t']\).
Fix a constant admissible control \(\delta_0\in A\), and write
\(L_s^0=L_s^{t,\ell,\delta_0}\) and \(M_s^0=M_s^{t,\ell,m,\delta_0}\).
The DPP, evaluated at this fixed control on \([t,t']\), yields
\begin{align*}
v(t,\ell,m)-v(t',\ell,m)
&\le
\EE\bigl[
v(t',L_{t'}^0,M_{t'}^0)
-v(t',\ell,m)
\bigr]\\
&\quad+\EE\left[\int_t^{t'} f(s,L_s^0,\delta_0)\dd s\right]\\
&\quad+\kappa\EE\left[\int_t^{t'}\varphi'(L_s^0)\dd M_s^0\right]\\
&\quad
-\EE\left[\int_t^{t'}\Pi_s(B_s-L_s^0)^+\,\dd s\right].
\end{align*}
The last term is nonpositive, and in any case has absolute value at most \(C_Kh\).
The same estimates as above, now with the fixed control \(\delta_0\), give
\[
v(t,\ell,m)-v(t',\ell,m)\le C_K h^{1/2}.
\]
The two inequalities prove the claim.
\end{proof}

\begin{lemma}[Polynomial growth of the value function]\label{le:polynomial-growth-v6}
If \(|\varphi(r)|\le C_\varphi(1+|r|^q)\) for some \(q\ge1\), then there is a constant \(C>0\), depending only on the model data, the bounds on \(B,\Pi\), and \(\varphi\), such that
\[
|v(t,\ell,m)|\le C(1+|m|^q),\qquad (t,\ell,m)\in[0,T]\times D.
\]
\end{lemma}

\begin{proof}
Since \(\ellbar\le \ell\le m\), it is enough to control polynomial moments of the reflected load by a constant times \(1+|m|^q\). Standard estimates for the reflected SDE, using bounded controls, bounded \(\Theta\), and the Lipschitz property of the Skorokhod map, give
\[
\EE\left[\sup_{t\le s\le T}|L_s^{t,\ell,\delta}|^q\right]\le C(1+|m|^q)
\]
uniformly over admissible controls \(\delta\). Also
\[
|M_T^{t,\ell,m,\delta}|\le C\left(1+|m|+\sup_{t\le s\le T}|L_s^{t,\ell,\delta}-\ell|\right),
\]
so \(\EE[|M_T^{t,\ell,m,\delta}|^q]\le C(1+|m|^q)\). The controls are bounded, \(f\) has at most linear growth in \(\ell\) uniformly over \(\delta\in A\), and the demand-response term has at most linear growth in \(L\) because \(B,\Pi\) are bounded. Proposition \ref{pr:peak-reduction-v6} rewrites the peak term as \(\kappa(\varphi(M_T)-\varphi(m))\). Hence
\[
\left| \EE\left[\int_t^T f(s,L_s^{t,\ell,u},u_s)\dd s
+\kappa\int_t^T \varphi'(L_s^{t,\ell,u})\dd M_s^{t,\ell,m,u}
-\int_t^T\Pi_s(B_s-L_s^{t,\ell,u})^+ds\right]\right|
\le C(1+|m|^q)
\]
for every admissible control \(u\). Taking the infimum gives the lower bound, and evaluating the cost at any fixed admissible constant control gives the upper bound.
\end{proof}

\section{Dynamic Programming Principle}\label{sec:dpp}

While a direct proof is possible, we choose to rely on an existing result proven for stochastic control problems for which the cost minimization is for a function of the terminal value of the state.
In order to use such a result, we introduce the accumulated-cost process which reads:
\begin{equation}
\label{fo:augmented}
Y_s=y+\int_t^s f(r,L_r^{t,\ell,\delta},\delta_r)\dd r+\kappa\int_t^s \varphi'(L_r^{t,\ell,\delta})\dd M_r^{t,\ell,m,\delta}
-\int_t^s\Pi_r(B_r-L_r^{t,\ell,\delta})^+dr.
\end{equation}
The merit of this introduction is to recast the running-cost problem as an optimal control problem over the $3$-dimensional process $(L_t, M_t,Y_t)$ with only a terminal cost. But first, we note the regularity property:

\begin{lemma}[Accumulated-cost stability]
\label{le:Y_Lip}
Using a prime in the notation of \eqref{fo:augmented} for load \(L^{t,\ell'}\) and running maximum \(M^{t,\ell',m'}\), and still with the same ordinary control and Brownian motion, we have
\[
\EE[|Y_T-Y_T'|]\leq C_K(|\ell-\ell'|+|m-m'|)
\]
when \((\ell,m)\) and \((\ell',m')\) range over a fixed compact subset \(K\subset D\). The constant \(C_K\) depends on \(K\), the model data, including the bounds on $B$ and $\Pi$, and the local growth constants of \(\varphi\).
\end{lemma}

\begin{proof}
The \(dr\)-part is controlled by the fact that \(f\) is linear in \(\ell\), so that
\[
|f(r,L_r,u_r)-f(r,L_r',u_r)|\leq C|L_r-L_r'|.
\]
The demand-response contribution satisfies the same local Lipschitz estimate, since
\[
\left|\Pi_r(B_r-L_r)^+-\Pi_r(B_r-L_r')^+\right|
\le \|\Pi\|_\infty |L_r-L_r'| .
\]
For the peak term, Proposition \ref{pr:peak-reduction-v6} gives
\[
\int_t^s \varphi'(L_r)\,dM_r=\varphi(M_s)-\varphi(m).
\]
The terminal maxima are Lipschitz in the initial state by Proposition \ref{pr:rmax}. Since the reflected load has finite moments of all orders under bounded controls, the polynomial growth and local Lipschitz bounds on \(\varphi\) imply
\[
\EE\left[
|\varphi(M_s)-\varphi(M_s')|+|\varphi(m)-\varphi(m')|
\right]\le C_K(|\ell-\ell'|+|m-m'|).
\]
Combining the two estimates proves the claim.
\end{proof}

We restate the problem on the canonical path space for ease of reference.

\begin{proposition}[Shifted control stability]\label{pr:control-dpp-construction-v6}
The admissible class can be realized on the shifted canonical space so that:
\begin{itemize}[leftmargin=18pt]
\item if \(\delta\in\cA_t\) and \(\tau\) is a stopping time, the restricted/shifted control \(\delta^\tau\) belongs to the corresponding continuation class after \(\tau\);
\item if a control is used before \(\tau\) and a measurable family of continuation controls is selected as a function of \((\tau,L_\tau,M_\tau,Y_\tau)\), then the pasted control is again admissible;
\item the controlled reflected system has the flow property after \(\tau\): conditionally on \(\mathcal F_\tau\), the continuation depends on the past only through \((\tau,L_\tau,M_\tau,Y_\tau)\) and the shifted continuation control.
\end{itemize}
\end{proposition}

\begin{proof}
For each starting time \(t\), let \(\Omega^t=C([t,T];\RR)\) be the canonical Brownian path space, with canonical process \(W^t\), natural filtration completed in the usual way, and Wiener measure. A control on \([t,T]\) is represented by a progressively measurable process with values in the compact set \(A\). The peak cost is a path functional of \(L\) and \(M\), so no representative of the control on the support of \(dM\) is needed.

If \(\tau\) is a stopping time, the shifted canonical path is
\[
\omega^\tau_s=\omega_{\tau(\omega)+s}-\omega_{\tau(\omega)},\qquad 0\le s\le T-\tau(\omega).
\]
The shifted control is defined by
\[
\delta^\tau_s(\omega)=\delta_{\tau(\omega)+s}(\omega)
\]
for \(0\le s\le T-\tau(\omega)\). Progressive measurability is preserved under this shift, up to the usual identification of the post-\(\tau\) filtration with the filtration generated by the shifted Brownian path. Hence restriction after \(\tau\) preserves admissibility.

For pasting, let \(\delta^0\) be the pre-\(\tau\) control and let \((s,\ell',m',y')\mapsto \delta^{s,\ell',m',y'}\) be a universally measurable selection of continuation controls. Define the pasted control by using \(\delta^0\) before \(\tau\) and, after \(\tau\), the shifted representative \(\delta^{\tau,L_\tau,M_\tau,Y_\tau}\). Progressive measurability follows from the standard measurable-pasting argument for stochastic control on canonical spaces. The compactness of \(A\), together with the spatial continuity of Proposition \ref{pr:spatial-lip-v5}, gives measurable \(\varepsilon\)-optimal selections by the Borel-space measurable-selection theorem of Bertsekas--Shreve \cite{BertsekasShreve1978}; alternatively, the weak-DPP theorem of Bouchard--Touzi \cite{BouchardTouzi2011} gives the required dynamic-programming inequalities without making this selection explicit. We use the explicit selection only to display the pasting argument.

Finally, pathwise uniqueness for the one-dimensional reflected SDE and the pathwise definition of \(M\) and \(Y\) imply the flow property. Indeed, after time \(\tau\), the shifted reflected equation starts from \((L_\tau,M_\tau,Y_\tau)\), is driven by the shifted Brownian motion and the shifted control, and produces the same continuation path as the original system. This proves the stated stability properties of the admissible class.
\end{proof}

The augmented value function is
\[
\widetilde V(t,\ell,m,y)=\inf_{\delta\in\cA_t}\EE[Y_T^{t,\ell,m,y,\delta}]
=y+v(t,\ell,m).
\]

\begin{theorem}[Dynamic programming principle]\label{thm:dpp-v6}
For any starting point \((t,\ell,m)\in[0,T]\times D\) and any stopping time \(\tau\in[t,T]\),
\begin{equation*}
\begin{split}
&v(t,\ell,m)=\inf_{\delta\in\cA_t}\EE\Bigl[
\int_t^\tau f(s,L_s^{t,\ell,\delta},\delta_s)\dd s+
\kappa\int_t^\tau \varphi'(L_s^{t,\ell,\delta})\dd M_s^{t,\ell,m,\delta}\\
&\hskip 95pt
-\int_t^\tau\Pi_s(B_s-L_s^{t,\ell,\delta})^+ds+
 v(\tau,L_\tau^{t,\ell,\delta},M_\tau^{t,\ell,m,\delta})\Bigr].
\end{split}
\end{equation*}
\end{theorem}

\begin{proof}
The point of introducing the accumulated-cost coordinate is that the problem can be written as a terminal-cost problem for the augmented state \((L,M,Y)\). Proposition \ref{pr:control-dpp-construction-v6} verifies the restriction, pasting, flow, and measurable-selection hypotheses needed for the weak dynamic programming theorem on the canonical space, as in \cite{BouchardTouzi2011}. Therefore the augmented value satisfies
\[
\widetilde V(t,\ell,m,0)
=\inf_{\delta\in\cA_t}\EE\left[
\widetilde V\bigl(\tau,L_\tau^{t,\ell,\delta},M_\tau^{t,\ell,m,\delta},Y_\tau^{t,\ell,m,0,\delta}\bigr)
\right].
\]
Substituting \(\widetilde V(s,\ell',m',y')=y'+v(s,\ell',m')\) and the definition of \(Y_s\) in \eqref{fo:augmented} gives the displayed formula. We spell out the two DPP inequalities.

Fix \(\delta\in\cA_t\). By the flow property of the reflected SDE and the additivity of the accumulated-cost coordinate \eqref{fo:augmented},
\[
Y_T^{t,\ell,m,0,\delta}
=Y_\tau^{t,\ell,m,0,\delta}
+Y_T^{\tau,L_\tau^{t,\ell,\delta},M_\tau^{t,\ell,m,\delta},0,\delta^\tau},
\]
where \(\delta^\tau\) is the shifted continuation control after \(\tau\). Taking conditional expectation and using the definition of \(v\) gives
\[
\EE\left[Y_T^{t,\ell,m,0,\delta}\mid\mathcal F_\tau\right]
\ge
Y_\tau^{t,\ell,m,0,\delta}
+v(\tau,L_\tau^{t,\ell,\delta},M_\tau^{t,\ell,m,\delta}).
\]
Indeed, by the flow property, the conditional continuation cost after \(\tau\) is the cost of the shifted control \(\delta^\tau\) started from \((\tau,L_\tau^{t,\ell,\delta},M_\tau^{t,\ell,m,\delta})\), while \(v\) is the infimum over all such admissible shifted continuation controls.
Taking expectations and then the infimum over \(\delta\) yields
\begin{equation*}
\begin{split}
&v(t,\ell,m)\ge
\inf_{\delta\in\cA_t}\EE\Bigl[
\int_t^\tau f(s,L_s^{t,\ell,\delta},\delta_s)\dd s+
\kappa\int_t^\tau \varphi'(L_s^{t,\ell,\delta})\dd M_s^{t,\ell,m,\delta}\\
&\hskip 95pt
-\int_t^\tau\Pi_s(B_s-L_s^{t,\ell,\delta})^+ds+
v(\tau,L_\tau^{t,\ell,\delta},M_\tau^{t,\ell,m,\delta})\Bigr].
\end{split}
\end{equation*}

For the reverse inequality, let \(\varepsilon>0\) and fix a control \(\delta\in\cA_t\) on the interval before \(\tau\). By Proposition \ref{pr:control-dpp-construction-v6}, for each possible state \((s,\ell',m')\) at time \(\tau\) one can choose a continuation control whose expected cost is within \(\varepsilon\) of \(v(s,\ell',m')\), measurably in \((s,\ell',m')\). Pasting this selected continuation after \(\tau\) to the fixed pre-\(\tau\) control gives an admissible control. Taking the infimum over the pre-\(\tau\) control after this pasting yields
\begin{equation*}
\begin{split}
&v(t,\ell,m)\le
\inf_{\delta\in\cA_t}\EE\Bigl[
\int_t^\tau f(s,L_s^{t,\ell,\delta},\delta_s)\dd s+
\kappa\int_t^\tau \varphi'(L_s^{t,\ell,\delta})\dd M_s^{t,\ell,m,\delta}\\
&\hskip 95pt
-\int_t^\tau\Pi_s(B_s-L_s^{t,\ell,\delta})^+ds+
v(\tau,L_\tau^{t,\ell,\delta},M_\tau^{t,\ell,m,\delta})\Bigr]+\varepsilon.
\end{split}
\end{equation*}
Letting \(\varepsilon\downarrow0\) proves the asserted identity.
\end{proof}

\section{Hamilton-Jacobi-Bellman (HJB) System}
We now formulate the corresponding HJB equation.

\subsection{Formal HJB System}

We first identify a local pointwise formulation of the system. We define the reduced Hamiltonian by:
\begin{equation}
\label{fo:Hamiltonian}
H(t,\ell,p)=\min_{u\in A}\left\{\frac{\beta}{2}u^2+(p+a+b\ell-\alpha_0)u\right\},
\end{equation}
We call it reduced because it does not involve the terms which are independent of the control. The argument of the minimization is obtained by projecting the unconstrained minimizer \(u^*=-(p+a+b\ell-\alpha_0)/\beta\) onto \(A\). Since $A$ is compact and convex, \(H\) is continuous and locally Lipschitz in \((\ell,p)\).
The following equations are formal consequences of the DPP obtained by applying It\^o's formula on a short time interval, dividing by the length of the interval, and sending the interval length to zero. The rigorous interpretation is the viscosity formulation given in the next sections.

\subsubsection*{Interior equation}

For $\ellbar<\ell<m$, take $\tau$ before the first time $\{L=M\}$ or $\{L=\ellbar\}$. On $[t,\tau]$, $dM_s=0$, so the DPP reduces to an ordinary controlled-diffusion identity with $m$ frozen:
\begin{equation}
\label{fo:HJB}
v_t-\lambda(\ell-\Theta_t)v_\ell+\frac{\sigma^2}{2}v_{\ell\ell}
+H(t,\ell,v_\ell)-\Pi_t(B_t-\ell)^+=0,
\end{equation}

\subsubsection*{Lower reflecting boundary}

At $\ell=\ellbar$, $\ell<m$: the term \(v_\ell\dd K_s\) in It\^o's formula, supported on $\{L_s=\ellbar\}$, gives the boundary condition:
\begin{equation}
\label{fo:HJB_BC1}
v_\ell(t,\ellbar,m)=0.
\end{equation}

\subsubsection*{Running-maximum diagonal boundary}

At $\ell=m$, It\^o's formula contributes \(v_m\dd M_s\). The peak cost contributes
\(\kappa\varphi'(L_s)\dd M_s=\kappa\varphi'(m)\dd M_s\) on the diagonal, giving the boundary condition
\begin{equation}
\label{fo:HJB_BC2}
v_m(t,m,m)+\kappa\varphi'(m)=0.
\end{equation}

\vskip 4pt
In the viscosity formulation below we use an operator \(G\), so that the subsolution condition is written as \(G\le0\) and the supersolution condition as \(G\ge0\).

\subsection{Viscosity Solution Formulation}
\label{sec:visc-form}

We introduce the operator symbol
\[
G(t,\ell,m,q,p,X)
=-q+\lambda(\ell-\Theta_t)p-\frac{\sigma^2}{2}X-H(t,\ell,p)+\Pi_t(B_t-\ell)^+,
\]
and the boundary operators:
\[
B_0(p)=-p,\qquad
B_M(m,r)=-(r+\kappa\varphi'(m)).
\]
We use test functions which are \(C^{1,2,1}\), namely once continuously differentiable in \(t\), twice continuously differentiable in \(\ell\), and once continuously differentiable in \(m\).

\begin{definition}[Viscosity solution]
An upper semicontinuous function \(u\) with \(u(T,\cdot)\leq0\) is said to be a viscosity subsolution if, whenever \(\psi\in C^{1,2,1}\) touches \(u\) from above at \((t,\ell,m)\), \(t<T\), the following relaxed conditions hold:
\[
G(\cdots)\le0,\qquad
\min\{G(\cdots),B_0(\psi_\ell)\}\le0,\qquad
\min\{G(\cdots),B_M(m,\psi_m)\}\le0,
\]
respectively in the interior, at the lower reflecting boundary, and at the running-maximum diagonal; at the corner both boundary operators are included in the minimum.

A lower semicontinuous function \(w\) with \(w(T,\cdot)\geq0\) is said to be a viscosity supersolution if, whenever \(\psi\in C^{1,2,1}\) touches \(w\) from below at \((t,\ell,m)\), \(t<T\), the corresponding relaxed inequalities hold:
\[
G(\cdots)\ge0,\qquad
\max\{G(\cdots),B_0(\psi_\ell)\}\ge0,\qquad
\max\{G(\cdots),B_M(m,\psi_m)\}\ge0,
\]
respectively in the interior, at the lower reflecting boundary, and at the running-maximum diagonal; at the corner the maximum is taken over all active operators.

A continuous function \(\phi\) is said to be a viscosity solution if it is both a viscosity subsolution and a viscosity supersolution.
\end{definition}

\begin{remark}
At $(\ellbar,\ellbar)$ both boundary conditions apply. The relaxed min--max convention is the standard viscosity way to encode reflected boundary conditions at a corner.
\end{remark}

\section{Viscosity Solutions of the HJB System}
\label{sec:viscosity}

\subsection{The Value Function as a Viscosity Solution}

\begin{theorem}[Viscosity characterization]
The value function \(v\) is a viscosity solution of the HJB system with interior equation \(G(\cdots)=0\), lower boundary operator \(B_0\), diagonal boundary operator \(B_M\), and terminal condition \(v(T,\cdot)=0\). It is unique in the comparison class specified in Theorem~\ref{thm:comparison-v5}.
\end{theorem}

\begin{proof}
Uniqueness is a consequence of the comparison principle proven in the following subsection.
The terminal condition follows from the definition of the payoff. The spatial Lipschitz estimate and Proposition \ref{pr:time-holder-v6} give the continuity needed to use ordinary test functions rather than semicontinuous envelopes.

Let \(z_0=(t_0,\ell_0,m_0)\), with \(t_0<T\), and let \(\psi\in C^{1,2,1}\). For \(h>0\) and a control \(\delta\), set
\[
\tau_h^\delta=\inf\{s\ge t_0: |L_s^{t_0,\ell_0,\delta}-\ell_0|+|M_s^{t_0,\ell_0,m_0,\delta}-m_0|\ge \rho\}\wedge(t_0+h)
\]
for a small \(\rho>0\). By Lemma \ref{le:short-time-state-v5} and Markov's inequality,
\[
\mathbb P[\tau_h^\delta<t_0+h]\le \rho^{-1}C h^{1/2}.
\]
Consequently
\[
0\le h-\EE[\tau_h^\delta-t_0]\le h\,\mathbb P(\tau_h^\delta<t_0+h)=o(h),
\]
uniformly over bounded controls. Applying It\^o's formula on \([t_0,\tau_h^\delta]\) gives
\[
\begin{aligned}
\psi(\tau_h^\delta,L_{\tau_h^\delta},M_{\tau_h^\delta})-\psi(t_0,\ell_0,m_0)
&=\int_{t_0}^{\tau_h^\delta}
\left(\psi_t+[-\lambda(L_s-\Theta_s)+\delta_s]\psi_\ell
+\frac{\sigma^2}{2}\psi_{\ell\ell}\right)(s,L_s,M_s)\dd s\\
&\quad+\int_{t_0}^{\tau_h^\delta}\psi_\ell(s,L_s,M_s)\dd K_s
+\int_{t_0}^{\tau_h^\delta}\psi_m(s,L_s,M_s)\dd M_s\\
&\quad+\sigma\int_{t_0}^{\tau_h^\delta}\psi_\ell(s,L_s,M_s)\dd W_s .
\end{aligned}
\]
The stochastic integral has expectation zero by localization and the boundedness of the derivatives of \(\psi\) on the stopped neighborhood.

We first prove the subsolution inequality. Suppose that \(\psi\) touches \(v\) from above at \(z_0\). If the relaxed subsolution condition failed at \(z_0\), then all three active operators would be strictly positive. Equivalently, for some \(\eta>0\), after reducing \(\rho\) if necessary, one can choose a constant ordinary control \(u_0\in A\), attaining the minimum in the definition of \(H\) at \(z_0\) up to \(\eta\), such that the \(ds\)-integrand
\[
\psi_t+[-\lambda(\ell-\Theta_t)+u_0]\psi_\ell+\frac{\sigma^2}{2}\psi_{\ell\ell}+f(t,\ell,u_0)-\Pi_t(B_t-\ell)^+
\]
is strictly negative at \(z_0\). By continuity of \(\psi\) and its derivatives, of  \(\Theta\), $\Pi$, $B$,  \(f\), and \(\varphi'\), reducing \(\rho\) if necessary, one can ensure that this \(ds\)-integrand is at most \(-\eta\) throughout the stopped neighborhood. If the lower boundary is active, the failed boundary condition gives \(\psi_\ell\le -\eta\) on the lower boundary in the same neighborhood. Since \(dK_s\ge0\) and \(dK_s\) is supported on this boundary, the term \(\psi_\ell\,dK_s\) is then nonpositive. If the diagonal is active, the failed diagonal condition gives \(\psi_m+\kappa\varphi'(m')\le-\eta\) for the current maximum coordinate \(m'\) on the diagonal in the same neighborhood. Since \(dM_s\ge0\), \(dM_s\) is supported on the diagonal, and \(L_s=M_s\) there, the term \((\psi_m+\kappa\varphi'(L_s))\,dM_s\) is nonpositive. Thus the finite-variation contributions are nonpositive and the \(ds\)-contribution is bounded above by \(-\eta(\tau_h^{u_0}-t_0)\).

The DPP applied to the constant control \(u_0\) on \([t_0,\tau_h^{u_0}]\), together with \(v\le\psi\) and equality at \(z_0\), yields
\[
\begin{aligned}
&0\le
\EE\Bigl[
\int_{t_0}^{\tau_h^{u_0}} f(s,L_s,u_0)\dd s
+\kappa\int_{t_0}^{\tau_h^{u_0}}\varphi'(L_s)\dd M_s -\int_{t_0}^{\tau_h^{u_0}}\Pi_s(B_s-L_s)^+ds\\
&\hskip 95pt
+\psi(\tau_h^{u_0},L_{\tau_h^{u_0}},M_{\tau_h^{u_0}})
-\psi(t_0,\ell_0,m_0)
\Bigr].
\end{aligned}
\]
Using the preceding It\^o expansion, the right-hand side is at most
\(-\eta\EE[\tau_h^{u_0}-t_0]+o(h)\), which is negative for \(h\) small. This contradiction proves the subsolution condition.

We next prove the supersolution inequality. Suppose that \(\psi\) touches \(v\) from below at \(z_0\). If the relaxed supersolution condition failed, then all active operators would be strictly negative. Thus, after reducing \(\rho\), there is \(\eta>0\) such that for every \(u\in A\) the corresponding \(ds\)-integrand with \(f(t,\ell,u)\) is at least \(\eta\) on the stopped neighborhood. If the lower boundary is active, then \(\psi_\ell\ge\eta\) on the lower boundary, so \(\psi_\ell\,dK_s\ge0\) because \(dK_s\ge0\). If the diagonal is active, then \(\psi_m+\kappa\varphi'(m')\ge\eta\) for the current maximum coordinate \(m'\) on the diagonal in the stopped neighborhood, so \((\psi_m+\kappa\varphi'(L_s))\,dM_s\ge0\) because \(dM_s\ge0\) and \(L_s=M_s\) on the support of \(dM_s\).

The preceding exit-time estimate is uniform over all bounded controls, so it applies to any control selected below. Let \(\varepsilon_h=o(h)\). By the lower DPP inequality, choose an \(\varepsilon_h\)-optimal control \(\delta^h\) on \([t_0,\tau_h^{\delta^h}]\). Since \(v\ge\psi\) and equality holds at \(z_0\),
\[
\begin{aligned}
&\varepsilon_h\ge
\EE\Bigl[
\int_{t_0}^{\tau_h^{\delta^h}} f(s,L_s,\delta_s^h)\dd s
+\kappa\int_{t_0}^{\tau_h^{\delta^h}}\varphi'(L_s)\dd M_s-\int_{t_0}^{\tau_h^{\delta^h}}\Pi_s(B_s-L_s)^+ds\\
&\hskip 95pt
+\psi(\tau_h^{\delta^h},L_{\tau_h^{\delta^h}},M_{\tau_h^{\delta^h}})
-\psi(t_0,\ell_0,m_0)
\Bigr].
\end{aligned}
\]
Since the \(ds\)-integrand is at least \(\eta\) for every \(u\in A\) in the stopped neighborhood, and the finite-variation terms are nonnegative, the It\^o expansion implies that the right-hand side is at least
\(\eta\EE[\tau_h^{\delta^h}-t_0]+o(h)\). The estimate \(\EE[\tau_h^{\delta^h}-t_0]=h+o(h)\) holds uniformly over bounded controls, hence for the selected control \(\delta^h\) in particular. This contradicts \(\varepsilon_h=o(h)\) for \(h\) small. Hence the supersolution condition holds.

At the corner \((\ellbar,\ellbar)\), both finite-variation terms may be active. The preceding contradiction arguments apply with both boundary operators included simultaneously in the minimum for subsolutions and the maximum for supersolutions. Therefore \(v\) is a viscosity solution in the relaxed boundary sense.
\end{proof}

\subsection{Comparison Principle}
For the purpose of this subsection we transform the HJB system described above.
 
We consider the new unknown function \(\widehat v=v+\kappa\varphi(m)\) and the coordinate change \(x=\ell-\ellbar\), \(r=m-\ell\). The new interior operator becomes
\[
\begin{aligned}
&\widehat P(t,x,r,q,p_x,p_r,X)
=q-\lambda(\ellbar+x-\Theta_t)(p_x-p_r)
+\frac{\sigma^2}{2}(X_{xx}-2X_{xr}+X_{rr})\\
&\hskip 125pt
+H(t,\ellbar+x,p_x-p_r)-\Pi_t\bigl(B_t-(\ellbar+x)\bigr)^+.
\end{aligned}
\]
Here \(X\) denotes the Hessian matrix in \((x,r)\). Accordingly, using the same sign convention as above, we set \(\widehat G=-\widehat P\). In this way, the diagonal boundary condition is homogenized, the state space is mapped to \(\overline{\mathbb R_+^2}\), the transformed reflection matrix is completely-\(\mathcal S\), and the transformed interior operator has a continuous Hamiltonian part and a degenerate-parabolic second-order part.

We recall the terminology used in the oblique reflection literature. A square matrix \(R\) is called completely-\(\mathcal S\) if every principal submatrix \(R_I\) is an \(\mathcal S\)-matrix; equivalently, for every nonempty index set \(I\), there exists a vector \(q_I\) with strictly positive components such that \(R_Iq_I\) has strictly positive components.

\vskip 4pt
Indeed, the change \(\widehat v=v+\kappa\varphi(m)\) turns the diagonal condition \(v_m+\kappa\varphi'(m)=0\) into \(\widehat v_r=0\) after the coordinate change \(x=\ell-\ellbar\), \(r=m-\ell\). The same coordinate change maps \(D\) onto \(\overline{\mathbb R_+^2}\). The derivative relations \(v_\ell=\widehat v_x-\widehat v_r\), \(v_{\ell\ell}=\widehat v_{xx}-2\widehat v_{xr}+\widehat v_{rr}\), and \(v_m+\kappa\varphi'(m)=\widehat v_r\) give the transformed operator \(\widehat P\). Since \(A\) is compact and \(H\) is continuous and locally Lipschitz, the Hamiltonian part of \(\widehat P\) is continuous on compact sets. The second-order term is
\[
\frac{\sigma^2}{2}(\partial_x-\partial_r)^2,
\]
which corresponds to the positive semidefinite diffusion matrix \(\sigma^2(1,-1)^\top(1,-1)\); hence \(\widehat G=-\widehat P\) is degenerate elliptic in the viscosity sense. Finally, the reflection fields transform to the columns of
\[
R=\begin{pmatrix}1&0\\ -1&1\end{pmatrix}.
\]
The one-dimensional principal submatrices plainly satisfy the \(\mathcal S\) condition. For the full matrix, one may take \(q=(1,2)^\top\), which gives \(Rq=(1,1)^\top>0\). Therefore \(R\) is completely-\(\mathcal S\).

\begin{remark}[Relation with terminal running-maximum costs]
The transformation \(\widehat v=v+\kappa\varphi(m)\) also identifies the present CP formulation with a terminal running-maximum penalty. Indeed, by Proposition~\ref{pr:peak-reduction-v6}, minimizing the cost with the incremental Stieltjes term
\(\kappa\int_t^T\varphi'(L_s)\,dM_s\) is equivalent, up to the current-state constant \(-\kappa\varphi(m)\), to minimizing the problem with terminal cost \(\kappa\varphi(M_T)\). Thus \(\widehat v\) has terminal condition
\[
\widehat v(T,\ell,m)=\kappa\varphi(m)
\]
and homogeneous diagonal condition \(\widehat v_r=0\) in the transformed variables. This is the precise point of contact with the running-maximum terminal-cost literature, such as \cite{RMaxControl4,RMaxControl5}. The additional features here are the finite-horizon data-center running cost, the lower reflecting load constraint and its corner interaction with the running-maximum boundary, and the demand-response term. This answers the question raised above: the singular-measure term can be treated as a terminal running-maximum cost depending only on \(M_T\), up to the current-state constant \(-\kappa\varphi(m)\). This simplifies the estimates and the numerical implementation, while the Stieltjes formulation remains useful for explaining the economics of peak creation.
\end{remark}

\begin{theorem}[Comparison principle via parabolic oblique comparison]\label{thm:comparison-v5}
Let \(u\) be an upper semicontinuous viscosity subsolution and \(w\) a lower semicontinuous viscosity supersolution of the HJB system, with \(u(T,\cdot)\leq w(T,\cdot)\), polynomial growth compatible with \(\varphi\), and local uniform continuity on compact subsets. Then \(u\leq w\) on \([0,T]\times D\). Within this comparison framework, \(v\) is the unique viscosity solution in this class.
\end{theorem}
The proof uses the cited oblique-boundary comparison machinery after the structural transformation below.

\begin{proof}
Set \(\widehat u=u+\kappa\varphi(m)\) and \(\widehat w=w+\kappa\varphi(m)\). This common change preserves the ordering at \(T\), the semicontinuity, the growth class, and local uniform continuity. Next use the coordinate change
\[
x=\ell-\ellbar,\qquad r=m-\ell,
\]
and denote by \(\widetilde u,\widetilde w\) the corresponding functions on the orthant. The change by \(\kappa\varphi(m)\) and the coordinate change preserve the viscosity inequalities in the usual test-function sense. By the structural transformation described above, the transformed problem is an oblique derivative problem on \(\overline{\mathbb R_+^2}\) with completely-\(\mathcal S\) reflection matrix, continuous Hamiltonian part, and degenerate elliptic operator \(\widehat G=-\widehat P\).

We now verify that the transformed problem falls under the parabolic oblique-derivative comparison theorem obtained by combining the Dupuis--Ishii corner construction for completely-\(\mathcal S\) reflection fields \cite{DupuisIshii1991} with the standard parabolic doubling-of-variables argument of \cite{CrandallIshiiLions1992}. The reflection matrix is completely-\(\mathcal S\) by the verification above; this is the corner compatibility condition used in the oblique-boundary construction. The operator \(\widehat G=-\widehat P\) is continuous in the jet variables on compact subsets and degenerate elliptic because the diffusion matrix \(\sigma^2(1,-1)^\top(1,-1)\) is positive semidefinite. To impose strict properness, use the standard exponential change \(\widetilde u_\gamma(t,x,r)=e^{-\gamma t}\widetilde u(t,x,r)\), \(\widetilde w_\gamma=e^{-\gamma t}\widetilde w\), with \(\gamma>0\). In the transformed variables the oblique boundary operators are homogeneous first-order operators, so this positive multiplication preserves the boundary inequalities; in the interior it adds the usual strictly monotone zero-order term. The change is removed by letting \(\gamma\downarrow0\) at the end.

For completeness, we recall how the parabolic and unbounded-domain reductions enter the argument. Suppose, to get a contradiction, that \(\sup(\widetilde u-\widetilde w)>0\). Because both functions have polynomial growth, introduce a coercive localization penalty, for \(N\) larger than the common growth order,
\[
\rho\bigl(1+x^2+r^2\bigr)^{N/2}
+\rho\bigl(1+y^2+s^2\bigr)^{N/2}
\]
and the usual doubling penalty
\[
\frac{|t-q|^2}{2\eta}
+\frac{|x-y|^2+|r-s|^2}{2\varepsilon}.
\]
For fixed \(\rho>0\), the maximum of the penalized difference is attained on a compact subset of \([0,T]\times\overline{\mathbb R_+^2}\), and the growth penalty lets one pass to the limit \(\rho\downarrow0\) after the comparison estimate. The time-doubling parameter \(\eta\) gives the parabolic jets in the sense of the user's guide \cite{CrandallIshiiLions1992}. At boundary points, the oblique test-function correction supplied by the completely-\(\mathcal S\) matrix in the Dupuis--Ishii construction makes the viscosity inequalities compatible with the two reflected faces. At the corner, the relaxed min--max convention used in Section~\ref{sec:visc-form} is precisely the viscosity statement that all active oblique boundary inequalities are included simultaneously. Sending first \(\varepsilon,\eta\downarrow0\), and then \(\rho,\gamma\downarrow0\), gives the contradiction to a positive maximum. Therefore \(\widetilde u\le \widetilde w\) on \([0,T]\times\overline{\mathbb R_+^2}\).

Reversing the coordinate and peak-penalty changes gives \(u\le w\) on \([0,T]\times D\). Thus the comparison conclusion relies on the cited oblique-boundary comparison machinery; the verification above records the structural hypotheses needed in the present problem, namely the completely-\(\mathcal S\) reflection matrix, degenerate ellipticity, properness after exponential change, and polynomial-growth localization. Lemma \ref{le:polynomial-growth-v6}, Proposition \ref{pr:spatial-lip-v5}, and Proposition \ref{pr:time-holder-v6} show that \(v\) belongs to the stated growth and continuity class. Applying the comparison result to two viscosity solutions in this class gives uniqueness.
\end{proof}

\section{Numerical Implications}

The viscosity formulation suggests the following structure for numerical schemes. In the interior, one discretizes the HJB operator \(P=-G\) with a monotone finite-difference approximation of the drift, diffusion, and Hamiltonian \(H\). At the lower reflecting boundary \(\ell=\ellbar\), the Neumann condition is imposed by a monotone one-sided approximation of \(v_\ell=0\). At the running-maximum diagonal \(\ell=m\), the diagonal condition is imposed as
\[
v_m+\kappa\varphi'(m)=0,
\]
again using a one-sided difference in the \(m\)-direction compatible with the wedge \(D\).
Equivalently, one can solve for the transformed value \(\widehat v=v+\kappa\varphi(m)\), impose the homogeneous diagonal condition \(\widehat v_r=0\), and then recover \(v=\widehat v-\kappa\varphi(m)\). This is often the cleaner implementation because the peak penalty is carried by the terminal condition \(\widehat v(T,\ell,m)=\kappa\varphi(m)\).

Using the comparison theorem above, a Barles--Souganidis type convergence result \cite{BarlesSouganidis1991} applies to any scheme which is monotone, stable, and consistent with the relaxed boundary formulation. Identifying such a scheme and proving these properties is not investigated in this paper; here, we merely identify the PDE and boundary conditions that such a scheme must approximate.


\subsection{Preliminary numerical experiments}
\label{sec:numerical-experiments-draft}

This section reports preliminary numerical experiments for the finite-horizon
single data center control problem.  The purpose is qualitative: we want to visualize
approximations of the value function and candidate feedback behavior. We compare three computational
routes suggested by the dynamic programming and HJB formulations.  We do not
claim a convergence theorem for the numerical methods in this section.

For the purpose of our implementation of the numerics, we assume that the time horizon is $T=1$ and that the mean-reverting load function $t\mapsto \Theta(t)$ is 
$$
\Theta(t)=\theta_s\bigl(1+0.1 \sin\bigl(2\pi t)\bigr),\qquad0\le t\le 1,
$$
for a scale constant $\theta_s=0.9$. Our qualitative results do not differ much from a constant function $\Theta(t)=0.9$. Our choice is consistent with what is found in the data-center literature and power-system planning reports. See for example \cite{DataCenterDR1, DataCenterDR2} describing standalone data centers as high-load-factor facilities with relatively flat and predictable electricity demand. Daily variation may still occur, but it is typically driven by workload scheduling, mixed-use occupancy, cooling demand, or AI workload bursts, rather than by the same deterministic shape as residential or regional aggregate load. We use a constant volatility $\sigma=0.15$ and $\underline\ell=0$ for the sake of definiteness. Negative values of $\underline\ell$ would make sense if the load was an effective load, for example the actual electricity demand of the data center corrected for the power it produces with its own generation, say solar panels. 
For the parameters of the running cost \eqref{fo:short_cost} we use $a=\alpha_0=0$ and $b=\beta=1$. We use the terminal-cost representation of the demand-charge term,
\[
  \kappa\int_t^T \varphi'(L_s)\,dM_s
  =
  \kappa\{\varphi(M_T)-\varphi(m)\},
\]
which allows us to avoid having to compute the integral with respect to \(dM\) by time discretization. We concentrate our efforts on two special cases for the peak penalty, namely
\[
 a)\quad \varphi(m)=\frac12m^2,\qquad\qquad \text{and}\qquad\qquad b)\quad \varphi(m)=m.
\]
We also assume that the baseline function $t\mapsto B_t$ is deterministic and constant, say equal to $B$. Finally, we assume that the function $\Pi_t$ is a multiple $\Pi\ge 0$ of a smoothed version of the indicator function of an interval, say $[0.25,0.65]$, delimiting a time window during which the Demand Response program is in effect.

\vskip 4pt
In order to visualize the results, we plot approximations of the value function and candidate controls for a fixed value \(m=m_0\) so that the value and feedback can
be plotted over the \((t,\ell)\)-plane.
We shall often refer to the benchmark case when the peak penalty is quadratic and there is no demand response reward, i.e. $\Pi=0$.

\subsection*{Incremental Dynamic Programming.}
The first algorithm uses a semi-Lagrangian scheme intended to approximate the viscosity solution of \eqref{fo:HJB} with the boundary conditions \eqref{fo:HJB_BC1} and \eqref{fo:HJB_BC2}.
We rely on a semi-Lagrangian DPP discretization.  For
nonlinear \(\varphi\), we use an incremental recursion for \(v\), with terminal
condition \(v(T,\ell,m)=0\) and one-step peak cost
\[
  \kappa\{\varphi(M_{n+1})-\varphi(M_n)\}.
\]
This formulation is numerically preferable to repeatedly interpolating the
convex terminal peak term in the transformed \(\widehat v=v+\kappa\varphi(m)\)
problem.  

In the quadratic benchmark, the current grid gives
\[
  v(0,\ell_0,m_0)\simeq 0.05637,
  \qquad
  \delta(0,\ell_0,m_0)\simeq -0.122.
\]
The discrete dynamic-programming recursion is internally consistent when its
stored feedback is evaluated with the same quadrature, interpolation, and
boundary projection conventions. 

For the cross-route quadratic DPP benchmark, the numerical parameters were
\[
T=1,\quad (\ell_0,m_0)=(0.65,0.80),\quad
\lambda=1.20,\quad \sigma=0.22,\quad
\Theta(t)=0.70+0.15\sin(2\pi t),
\]
with \(a=0.35\), \(b=0.30\), \(\alpha_0=0.70\),
\(\beta=0.60\), \(\kappa=2.00\), \(\Pi=0\), and
\(\delta\in[-0.80,0.80]\).  The incremental-DPP computation used
\(x=\ell-\underline\ell\) and \(r=m-\ell\) on
\([0,2.0]\times[0,2.0]\), with \(51\) grid points in each space
direction, \(56\) time steps, and \(51\) controls.  The one-step expectation
was approximated by a \(5\)-point Gaussian quadrature for the Euler endpoint
and a \(3\)-point Gauss--Legendre quadrature for the Brownian-bridge maximum,
with bilinear interpolation in \((x,r)\).  The lower reflecting boundary was
handled by the projected Euler update; no additional finite-difference
boundary projection was imposed in this particular incremental-DPP run.
With these conventions, the current grid gives
\[
  v(0,\ell_0,m_0)\simeq 0.05637,
  \qquad
  \delta(0,\ell_0,m_0)\simeq -0.122.
\]
The discrete dynamic-programming recursion is internally consistent when its
stored feedback is evaluated with the same quadrature and interpolation
kernel: the same-kernel policy evaluation gives the same initial value up to
machine precision.  An independent forward Monte Carlo evaluation of the same
stored feedback gives \(0.04120\), with standard error \(2.38\times10^{-4}\),
so the gap is approximately \(-1.52\times10^{-2}\).  This gap compares two
different numerical objects, namely the semi-Lagrangian quadrature/interpolation
kernel used by the DPP and a separate continuous-path Euler Monte Carlo
evaluation.  Thus the discrepancy has been quantified, but these numerical
values should not be presented as validation of the DPP scheme unless a common
evaluation protocol is used or the gap is further resolved.

As a pilot refinement check, we also ran a boundary-projected DPP variant on
\([0,1.8]^2\).  In the quadratic case, coupled grid-time-control refinements
gave
\[
\begin{array}{c|ccc}
(N_x,N_t,N_u) & (41,48,41) & (51,60,51) & (61,72,61)\\
\hline
v_0 & 0.08740 & 0.07773 & 0.07492
\end{array}
\]
respectively.  Matched-mesh domain enlargements from \(x_{\max}=r_{\max}=1.80\)
to \(2.16\) and \(2.52\), holding \(\Delta x=\Delta r=0.036\), left
\(v_0\) unchanged at the reported precision.  Comparing the level-2
unprojected and projected boundary treatments gave \(v_0=0.06846\) and
\(v_0=0.07492\), respectively.  These checks are reported only as pilot
reproducibility diagnostics; they are not used here as a proof of convergence
of the numerical scheme.

\subsection*{Deep Galerkin Value Approximation.}
The second method approximates the value function by a smooth neural network.
The parameters of the neural network are learned by minimizing, via Stochastic Gradient Descent (SGD), a loss function comprising the squared HJB residual, the terminal residual, and the
lower-boundary and diagonal-boundary residuals.  Because the network is
differentiable, it also gives a smooth feedback proxy through the minimizer of
the Hamiltonian.

Pure residual training fits the terminal and boundary conditions well, but in
the current experiments the derivative feedback is too weak near the initial
state.  We therefore also use an anchored Deep Galerkin model.  This model keeps
the PDE, terminal, and boundary residual losses, but adds value and policy
samples from the incremental DPP benchmark, with oversampling on the fixed
\(m=m_0\) surface and near the initial state.  It should be viewed as a smooth
visualization proxy, not as an independent solver.

In the quadratic benchmark, the anchored model gives
\[
  v_{\rm DGM}(0,\ell_0,m_0)\simeq 0.06458,
  \qquad
  \delta_{\rm DGM}(0,\ell_0,m_0)\simeq -0.183.
\]
On the interior fixed-\(m\) surface, its policy \(L^2\)-distance to the
incremental-DPP feedback is about \(0.03041\), compared with \(0.11808\) for the
residual-only DGM.  This makes it the most useful smooth value-function and
feedback proxy among the current neural value-function experiments.

\subsection*{Direct SGD Optimization}

The third method bypasses the characterization of the value function as the solution of a PDE. 
It attempts to solve directly the optimization problem over an approximate class of feedback controls.
We parameterize the feedback control directly
as a neural network \(\delta_\theta(t,L_t,M_t)\).  For each value of the network
parameters, we replace the objective function \eqref{fo:full_cost} by a Monte Carlo approximation 
computed from sample paths simulated by a projected Euler scheme at the lower boundary,
and we use the terminal peak cost
\(\kappa\{\varphi(M_T)-\varphi(m)\}\). We then use Stochastic Gradient Descent (SGD) to minimize this Monte Carlo approximation of the objective.

Since the proxy of the value function we plot is obtained by evaluating expected cost associated to the learned control by Monte Carlo, it yields an upper bound for the exact value function, the error coming from the approximation of the optimal control compounded by the Monte Carlo evaluation of the expected cost. Moreover, the surface plots of the proxies of the value functions suffer from the fact that they are computed for $m=m_0$ fixed which requires a special interpretation of the regime $\ell\ge m_0$. This is discussed in the comments of the figures below.

\vskip 6pt
The algorithm proceeds as follows.

\begin{itemize}
\item Choose a time horizon, say $T=1$ and a number of time steps, say $N_t=100$. The time subdivision will be given by $t_i=i \Delta_t$ for $i=0,1,\cdots,N_t$ where we set $\Delta_t=T/N_t$.
\item Choose a number of Monte Carlo samples, say $N_{sim}$, and generate a table $\epsilon^j_{i}$ of independent $N(0,1)$ random samples for $i=0,1,\cdots,N_t-1$ and $j=1,\cdots,N_{sim}$.
\item Choose an architecture for a deep neural network \(\delta_\theta(t,\ell,m)\) with the time variable restricted to the chosen time grid.
\item For each neural network parameter $\theta\in\Theta$
\begin{itemize}
\item For each simulation $j=1,\cdots,N_{sim}$, generate an admissible load path inductively for $i=0,1,\ldots,N_t-1$, starting from $\ell^j_{0}=\ell$ and \(m_0^j=m\), using the formula
$$
\widetilde\ell^j_{i+1}= \ell^j_i+\Delta_t\Bigl[- \lambda(\ell^j_i-\Theta(t_i))+\delta_\theta(t_i,\ell^j_i, m^j_i) \Bigr]+ \sigma \sqrt{\Delta_t} \epsilon^j_i,\qquad
\ell^j_{i+1}=\max\{\underline\ell,\widetilde\ell^j_{i+1}\}.
$$
where $m^j_{i+1}=\max\{m^j_i,\ell^j_{i+1}\}$, so that the running maximum starts from the prescribed initial value \(m\).
\item Compute the objective \eqref{fo:full_cost} for this sample path
$$
J_\theta^j= \Delta_t\sum_{i=0}^{N_t-1}\Bigl[
f(t_i,\ell^j_i,\delta_\theta(t_i,\ell^j_i,m^j_i))
-\Pi_{t_i}\bigl(B_{t_i}-\ell^j_i\bigr)^+\Bigr]
+\kappa \bigl[\varphi(m^j_{N_t})-\varphi(m)\bigr],
$$
where the running cost $f(t,\ell,\delta)$ is defined in \eqref{fo:short_cost}.
\end{itemize}
\item Compute the average $J_\theta$ of the resulting $N_{sim}$ objectives.
\item Run SGD to find a learned parameter $\hat\theta\in\Theta$ with a low empirical objective value $J_\theta$.
\item For selected values $m_0$, produce a surface plot of the learned control $(t,\ell)\mapsto \delta_{\hat\theta}(t,\ell,m_0)$.
\item For the selected values of $m_0$, produce a surface plot of the Monte Carlo policy-cost estimate $J_{\hat\theta}$.
\end{itemize}

\subsection*{Numerical Experiments}
We reproduce below some results of the implementation of the third method described above.
Unless overridden in the figure captions, the basic parameters used in these runs are summarized in Table \ref{ta:parameters}.

\begin{center}
\begin{table}[ht]
\centering
\small
\setlength{\tabcolsep}{4pt}
\caption{Numerical parameters used by default}
\label{ta:parameters}
\begin{tabular}{clrclr}
\toprule

Symbol & Description & Value&Symbol & Description & Value\\

\midrule

$T$ & Time horizon & 1.0&$\sigma$ & volatility & 0.15\\
$\lambda$ & Mean reversion coefficient & 1.0&$m_0$ & Initial running maximum & 1.0\\
$\underline\ell$ & Load lower bound & 0.0&$\bar\ell$ & Computational grid upper limit & 2.0\\
$\theta_s$ & $\Theta(t)$ scale factor & 0.9&$\beta$ & Control penalty coefficient & 1.0\\
$\underline\delta$ & Control lower bound & -1.0&$\bar\delta$ & Control upper bound & 1.0\\
$a$ & Running cost coefficient & 0.0&$b$ & Running cost coefficient & 1.0\\
$\kappa$ & Peak-charge penalty coefficient & 1.0&$\varphi$ & Peak-charge penalty type & "quadratic"\\
$\Pi$ & DR incentive rate& 1.0&$B$ & DR baseline load level & 0.95\\
\bottomrule

\end{tabular}
\end{table}
\end{center}

For the direct SGD optimization we use a \emph{small} neural network with $3$ features, input $x=(t,\ell,m)$, $n_\ell=1$ hidden layer, $n_u=8$ units in the layer, so $41$ learned parameters with $\tanh$ activation function. As given in the text above, we used $N_t=100$ time-discretization steps, $N_{sim}=256$ Monte Carlo samples for training and $N_{sim}=4096$ for final evaluation.

\subsection*{Numerical Results}

The plots and comments below use the neural architecture and the parameters given above, except where the captions specify overrides. For clarity, define
\[
V^{\hat\theta}(t,\ell,m):=J(t,\ell,m;\delta_{\hat\theta}),
\]
the expected cost of the learned policy. In general, $V^{\hat\theta}$ is a policy-value estimate and is not assumed to equal the optimal value function $v$. Indeed, for $0\le \ell\le m_0$, the displayed surfaces show $(t,\ell)\mapsto\delta_{\hat\theta}(t,\ell,m_0)$ and $(t,\ell)\mapsto V^{\hat\theta}(t,\ell,m_0)$. For $m_0\le \ell\le\bar\ell$, they instead show diagonal states, namely $(t,\ell)\mapsto\delta_{\hat\theta}(t,\ell,\ell)$ and $(t,\ell)\mapsto V^{\hat\theta}(t,\ell,\ell)$. Thus the full plotted surface is not a fixed-$m$ slice.

Since the original value \(v\) is written with the CP term in incremental form and there is no additional terminal payoff, the terminal condition is \(v(T,\ell,m)=0\). This is quite clear from the plots.

\begin{figure}[H]
  \centering
  \includegraphics[width=0.9\textwidth]{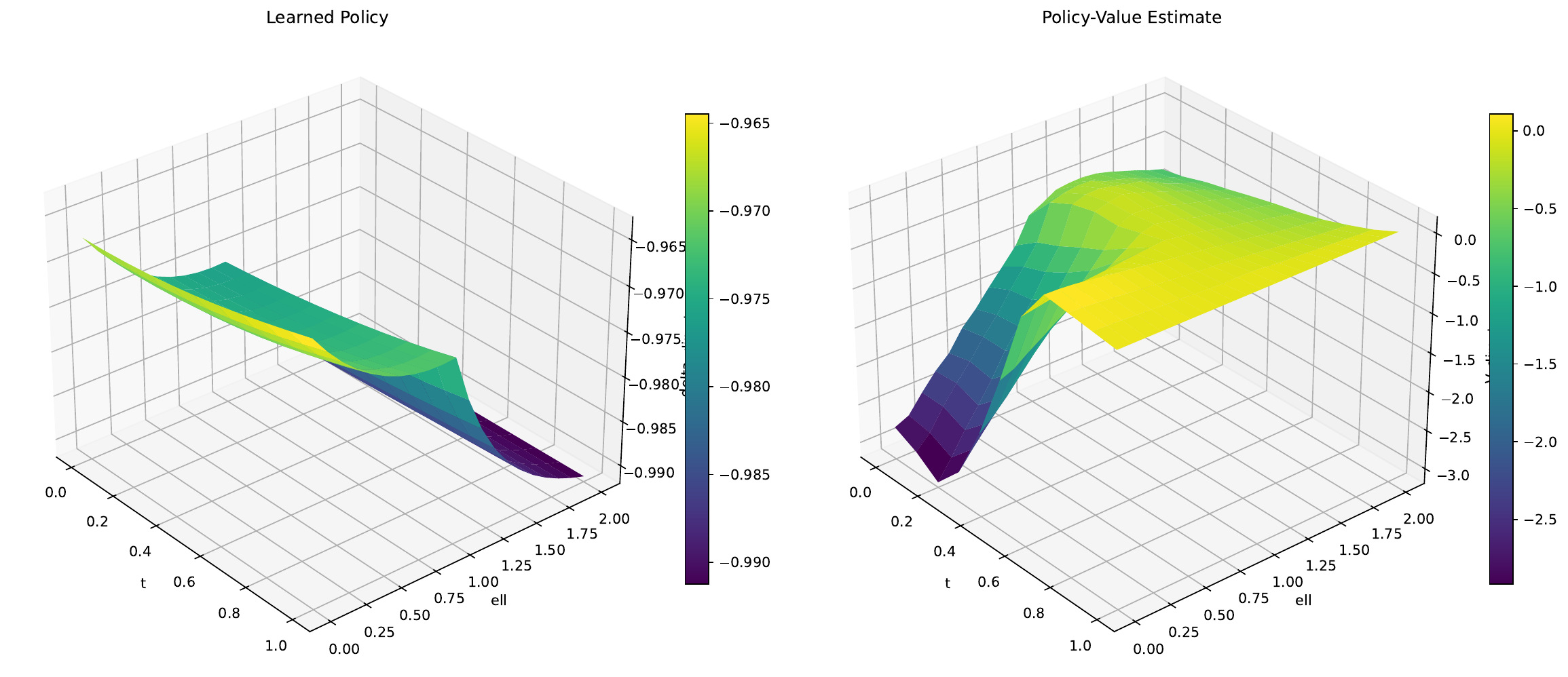}
  \caption{Learned feedback (left) and Monte Carlo policy-value estimate (right) computed with the parameters in Table \ref{ta:parameters} with the overrides $\kappa=10.0$, $m_0=1.0$ and $\Pi=10.0$.}
  \label{fig:kappa10.0m_0=1.0Pi=10.0}
\end{figure}

The left pane of Figure \ref{fig:kappa10.0m_0=1.0Pi=10.0} shows clearly the two regimes ($\ell\le m_0$ and $\ell\ge m_0$) one should expect given the explanations spelled out at the top of this subsection. 
This is an intentionally extreme DR case: the learned feedback is close to the lower control bound over most of the plotted region, which indicates that the DR reward dominates the reduced running cost and the CP penalty for these parameter values.
The estimated policy value is negative, a negative cost meaning an expected profit within this reduced-form objective.

\begin{figure}[H]
  \centering
  \includegraphics[width=0.9\textwidth]{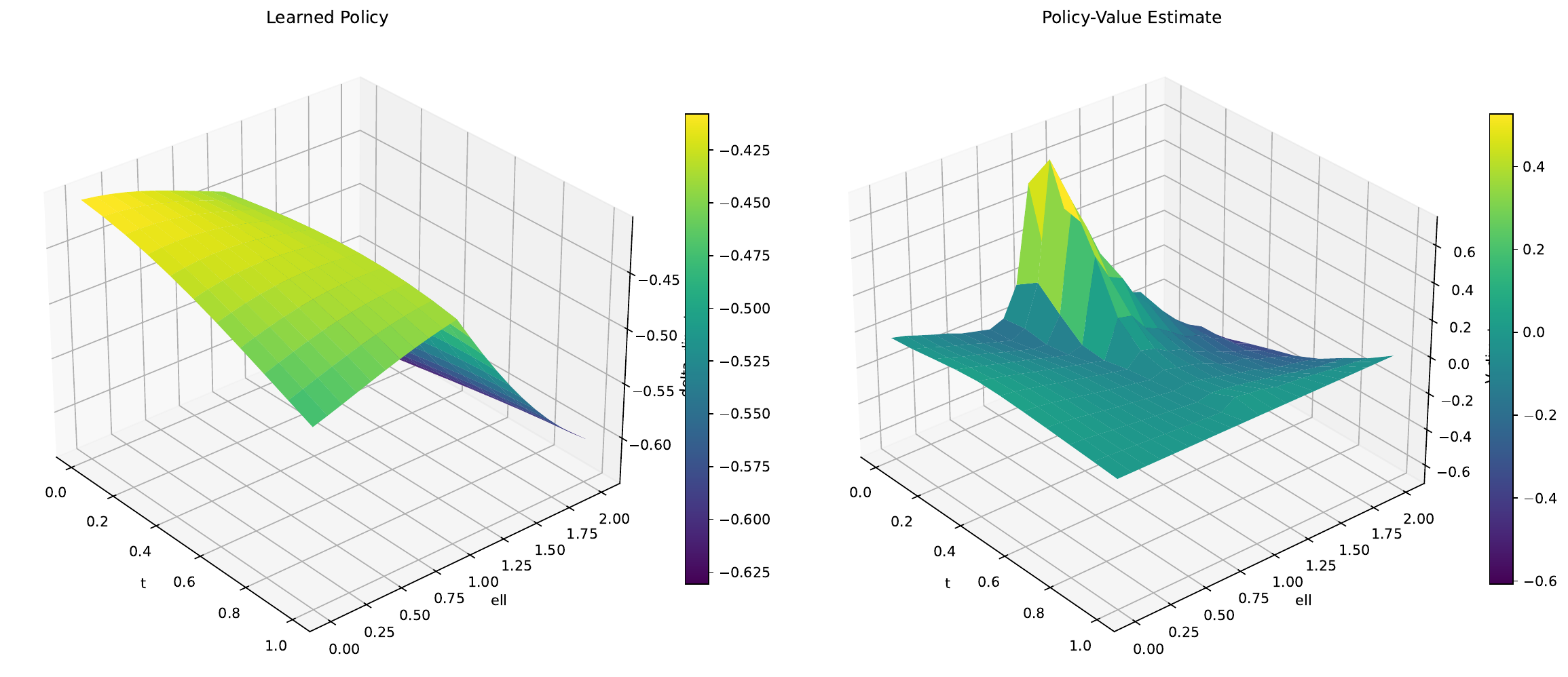}
  \caption{Learned feedback (left) and Monte Carlo policy-value estimate (right) computed with the parameters in Table \ref{ta:parameters} with the overrides $\kappa=10.0$, $m_0=1.0$ and $\Pi=0.0$.}
  \label{fig:kappa10.0m_0=1.0Pi=0.0}
\end{figure}

Figure \ref{fig:kappa10.0m_0=1.0Pi=0.0} provides plots for the same set of parameters in the absence of participation in the Demand Response (DR) program, namely after setting $\Pi=0$. The difference is striking. While the left pane still shows the sharp transition when $\ell$ crosses the threshold $m_0$, the right pane shows a large swath of values of $t$ and $\ell$ for which the estimated policy value is positive, meaning costs for the data center within this reduced-form objective. When the current maximum is already high, reducing the load does not undo the previously accumulated peak term; it only lowers the future running cost and reduces the probability of creating a still higher peak. This explains why the sign and magnitude of the feedback depend strongly on the position of \(\ell\) relative to the current maximum.

\vskip 4pt
Given that we did not model the major sources of revenue of the data center, the form of the cost is such that the main source of gain is the DR program. For that reason, the learned control $\hat\delta$ tends to keep the load as low as possible to increase this source of gain, and this is why $\hat\delta$ is negative in all the plots reproduced above. Figure \ref{fig:a=-10.0} below shows that changing dramatically the nature of the running cost $f(t,\ell,\delta)$ can produce regimes in which the learned control $\hat\delta$ is positive.

\begin{figure}[H]
  \centering
  \includegraphics[width=0.9\textwidth]{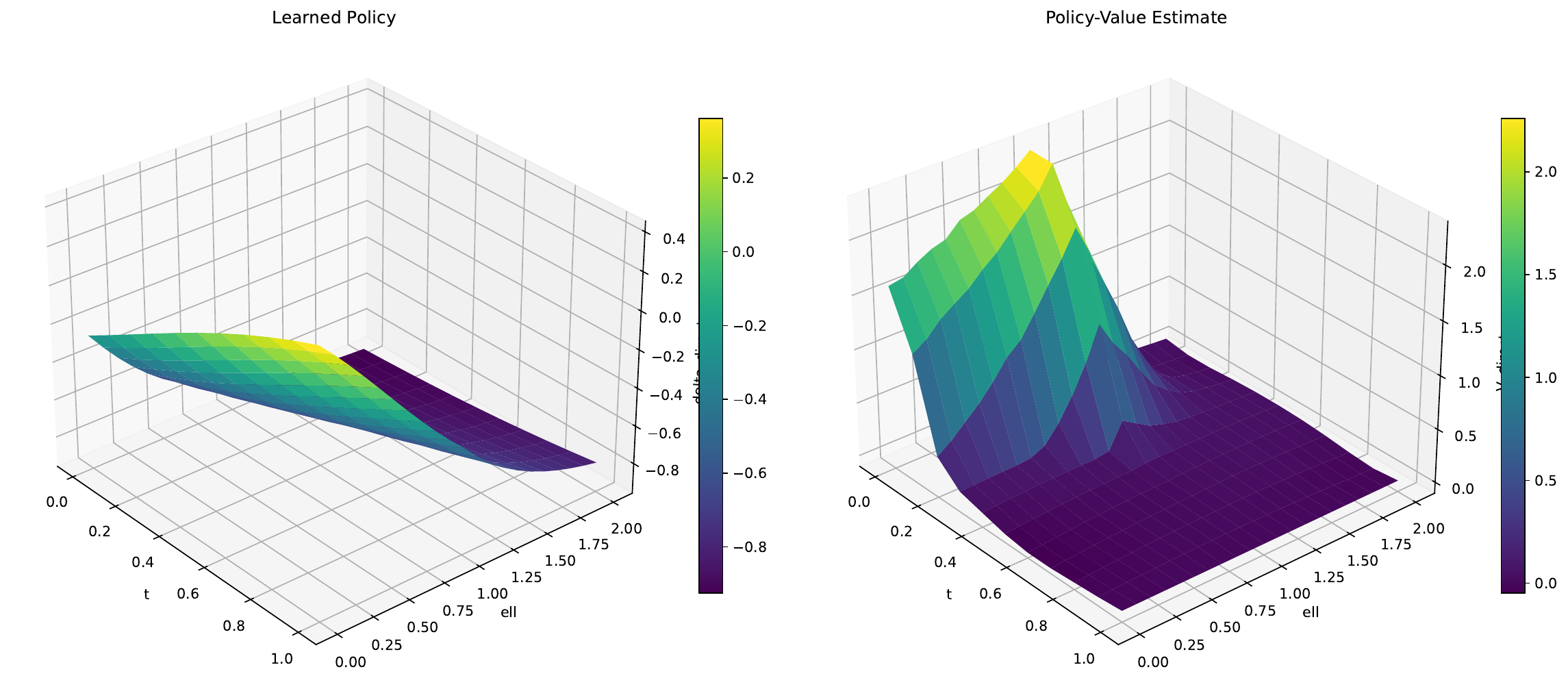}
  \caption{Learned feedback (left) and Monte Carlo policy-value estimate (right) computed with the parameters in Table \ref{ta:parameters} with the major override $a=-1.0$.}
  \label{fig:a=-10.0}
\end{figure}

\paragraph{Acknowledgements:}
The first named author (RC) was partially supported by the AFOSR grant FA-9550-23-1-0324.

\appendix
\section{Local Verification Argument}

For a smooth candidate \(\widehat V\) and an admissible control \(u\), It\^o's formula gives
\[
\dd\widehat V=\left(\widehat V_t+[-\lambda(L_s-\Theta_s)+u_s]\widehat V_\ell+
\frac{\sigma^2}{2}\widehat V_{\ell\ell}\right)\dd s
+\sigma\widehat V_\ell\dd W_s+\widehat V_\ell\dd K_s+\widehat V_m\dd M_s.
\]
Equivalently, applying It\^o's formula to \(\widehat V+Y\), with \(Y\) as in \eqref{fo:augmented}, gives the \(ds\)-coefficient
\[
\widehat V_t+[-\lambda(\ell-\Theta_t)+u]\widehat V_\ell+
\frac{\sigma^2}{2}\widehat V_{\ell\ell}+f(t,\ell,u)
-\Pi_t(B_t-\ell)^+,
\]
and the finite-variation coefficients \(\widehat V_\ell\) in front of \(dK_s\) and \(\widehat V_m+\kappa\varphi'(L_s)\) in front of \(dM_s\). The condition $\widehat V_\ell(t,\ellbar,m)=0$ kills the $\dd K_s$ term. Since \(dM_s\) is supported on $\{L_s=M_s\}$, the term \((\widehat V_m+\kappa\varphi'(L_s))\dd M_s\) gives
\[
\widehat V_m(t,m,m)+\kappa\varphi'(m)=0.
\]

\end{document}